\documentclass[
final, nomarks
]{dmtcs-episciences}

\usepackage[utf8]{inputenc}
\usepackage{subfigure}
\usepackage{amsfonts,amsmath,amssymb}
\newtheorem{theorem}{Theorem}[section]
\newtheorem{conjecture}[theorem]{Conjecture}
\newtheorem{proposition}[theorem]{Proposition}
\newtheorem{lemma}[theorem]{Lemma}

\newtheorem{corollary}[theorem]{Corollary}
\newtheorem{question}[theorem]{Question}
\newtheorem{example}[theorem]{Example}
\usepackage[round]{natbib}

\author{Mikl\'os B\'ona \affiliationmark{1}\thanks{Partially supported by Simons Collaboration Grant 421967.}
  \and  Michael Cory \affiliationmark{2}\thanks{Supported by a UF Scholars award.}
}
\title[Cyclic Permutations Avoiding Patterns]{Cyclic Permutations Avoiding Pairs of Patterns of Length Three}
\affiliation{
  University of Florida, USA}
\received{2018-12-5}
\revised{2019-10-16}
\accepted{2019-10-16}
\begin{document}
\publicationdetails{21}{2019}{2}{8}{5014}
\maketitle
\begin{abstract}
 We  enumerate  cyclic permutations avoiding two patterns of length three each
by providing explicit formulas for all but one of the pairs for which no such formulas were known. The pair $(123,231)$ proves
to be the most difficult of these pairs. We also prove a lower bound for the growth rate of the number of cyclic 
permutations that avoid a single pattern $q$, where $q$ is an element of a certain infinite family of patterns.
\end{abstract}

\section{Introduction}
The theory of {\em permutation patterns}
considers permutations as {\em linear orders}. That is, a permutation $p$ is simply a linear order $p_1p_2\cdots p_n$
of the integers $[n]=\{1,2,\cdots ,n\}$.  Let $p=p_1p_2\cdots p_n$ be a permutation, let $k<n$, and let $q=q_1q_2\cdots
q_k$ be another permutation. We say that $p$ {\em contains} $q$ as a pattern
if there exists a subsequence $1\leq i_1<i_2<\cdots <i_k\leq n$
 so that for all indices
$j$ and $r$, the inequality $q_j<q_r$ holds if and only if the inequality
$p_{i_j}<p_{i_r}$ holds. If $p$ does not contain $q$, then we say
that $p$ {\em avoids} $q$.  An exact formula for the number $S_n(q)$ of $q$-avoiding permutations of length $n$ is known for all
patterns $q$ of length three, and all patterns $q$ of length four, except 1324, and its reverse, 4231. There are numerous
other results on the growth rate of the sequences $S_n(q)$ as well. See \cite{vatter} for an overview of these results.

Questions about pattern avoidance  become much more difficult if we also consider permutations
as {\em elements of the symmetric group}, or even just bijections over the set $[n]$ that have a cycle decomposition.

In this paper, we study pattern avoiding permutations that {\em consist of a single cycle}, or, as we will call them,
{\em cyclic permutations}.  Let $C_n(q)$ be the number of cyclic permutations of length $n$ that avoid the pattern $q$. Similarly, let $C_n(q,q')$ be the
number of cyclic permutations that avoid both patterns $q$ and $q'$.   The problem of determining
$C_n(q)$ for any given pattern $q$ of length three was raised by Richard Stanley at the Permutation Patterns conference in 2007. No such formulas have been found. The main result of the paper will be an explicit
formula for the sequence $C_n(123,231)$ counting cyclic permutations that avoid both 123 and 231. We will also prove an explicit enumeration formula for the easier pair
$(123,132)$. Taken together with a result of Archer and Elizalde \cite{archer}, which was the first non-trivial result of the
field,  and some straightforward pairs that we handle in Section \ref{others},
this will complete the analysis of cyclic permutations avoiding any given pair of patterns of length three, except 
for the pair $(132,213)$. For that pair of patterns, an exact formula is still not known, but an upper bound has recently been proved
by Brice Huang \cite{huang}.

The cited results of \cite{archer} and \cite{huang}, and the results of this paper, enable us to make the following comparison. 
Let $q$ and $q'$ be two distinct patterns of length three each. Let 
$S_n(q,q')$ be the number of
{\em all} permutations of length $n$ that avoid both patterns $q$ and $q'$. See \cite{combperm} for exact enumeration formulas for the
numbers $S_{n}(q,q')$.  Using those formulas, 
\[\lim_{n\rightarrow \infty} \frac{C_{n}(q,q')}{S_{n}(q,q')} =0.\]

We end the paper by stating some open problems and conjectures. We solve a special case of one of the conjectures, proving that
if $q$ is an element of a certain infinite family of patterns, then $2C_n(q) \leq C_{n+1}(q) $ for $n\geq 2$. 

\section{The pair $(123,231)$}
In this section, we enumerate cyclic permutations that avoid both 123 and 231. This is the most difficult of the pairs we handle in this paper.  We start by proving a collection of structural properties of such permutations.
We will use some basic facts about inversions of permutations and conjugacy classes in the symmetric group. These facts can be found in many
introductory combinatorics textbooks, such as \cite{awalk}.

\subsection{Preliminary lemmas}
\subsubsection{Bounds on layer sizes}

First, we show what a typical cyclic permutation that avoids both 123 and 231 must look like. Recall that an involution is
a permutation whose square is the identity permutation. In other words, an involution is a permutation in which each
cycle is of length 1 or 2. 
\begin{lemma} \label{layers}
Let $p$ be a permutation  of length $n$ that avoids the patterns 123 and 231 and which is not an involution. Then there exist three positive integers
$a$, $b$, and $c$ so that $a+b+c=n$, and
\[p= n \ n-1 \cdots (n-a+1) \ \ b \ (b-1) \cdots 1 \ \ (b+c) \ (b+c-1) \cdots (b+1).\]
\end{lemma}

In other words, the lemma states that $p$ consists of three decreasing subsequences of consecutive integers in consecutive positions, namely, $p$ starts with a decreasing subsequence of its $a$ largest entries, then continues with a decreasing subsequence of its $b$ smallest entries, and then it ends in a decreasing subsequence of its $c$ remaining entries. These three decreasing subsequences will be called the {\em layers} of $p$.  For instance, if $n=9$, and
the layer lengths are $a=4$, $b=2$ and $c=3$, then $p=987621543$. See Figure \ref{firsttriple} for an illustration. 

\begin{figure}[htbp]
 \begin{center}
  \includegraphics[width=60mm]{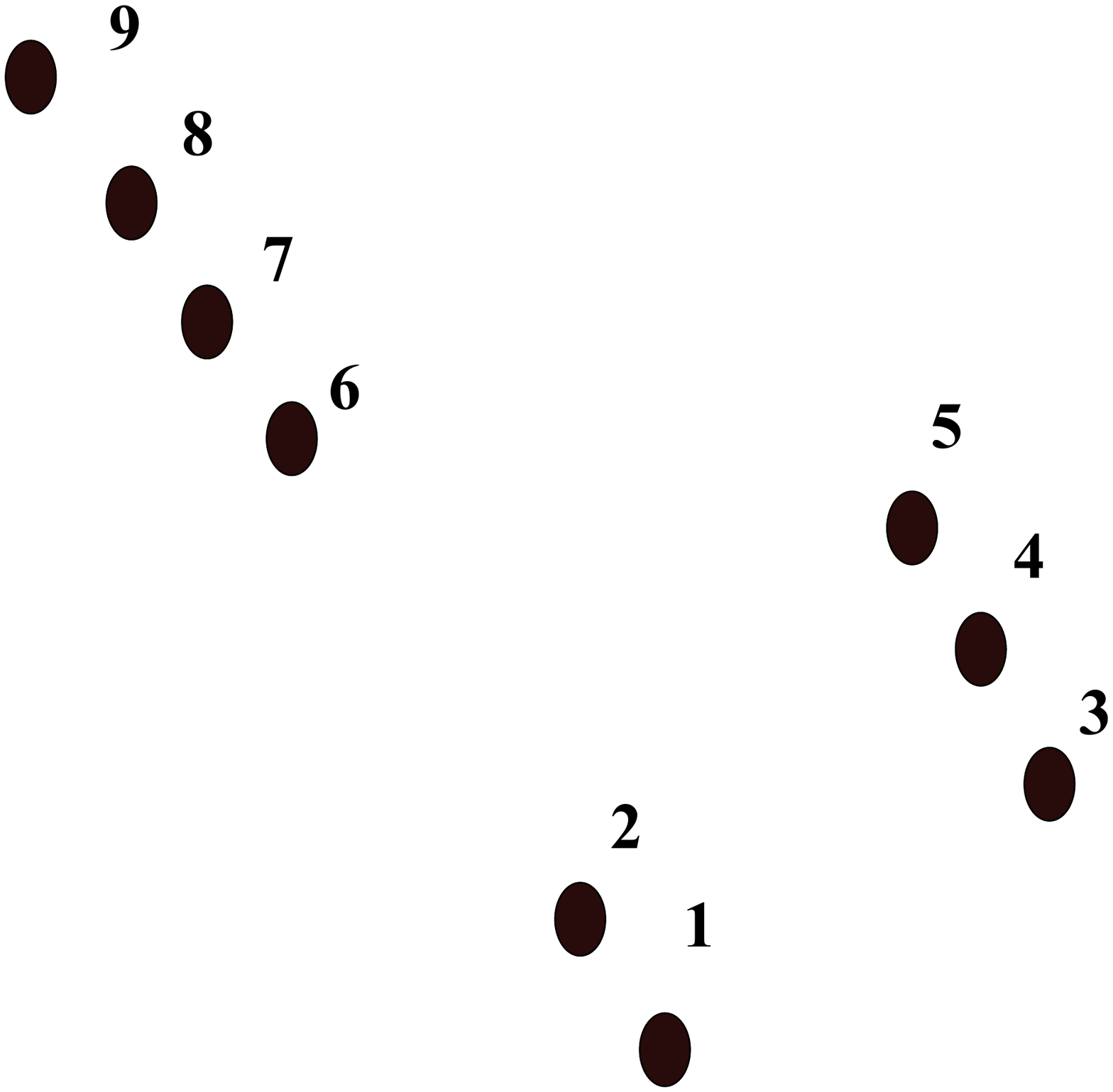}
  \label{firsttriple}
\caption{The permutation $p$ that belongs to the triple $(4,2,3)$. }
 \end{center}
\end{figure}

\begin{proof}
All entries preceding the entry $n$ have to be smaller than all entries following $n$ or a 231-pattern would be formed.
All entries preceding the entry $n$ must be in decreasing order or a 123-pattern would be formed.

If $n$ is not the leftmost entry, then this means that all entries on the right of $n$ must be in decreasing order, or  a 123-pattern is formed. So if $n$ is not the leftmost entry, then
\[p= i \ (i-1) \cdots 1 \ \ n \ (n-1) \cdots (i+1),\] but then $p$ is an involution.

That is, if $p$ is not an involution, then $p=p_1p_2\cdots p_n$ starts with the entry $p_1=n$. Let $a$ be the largest
integer so that we have $p_1p_2\cdots p_a=n(n-1)\cdots (n-a+1)$.  As $p$ is not an involution, it follows that $a\leq n-2$.
Repeating the argument of the first  paragraph of this proof for the remaining entries $\{1,2,\cdots ,n-a\}$ of $p$,
we see that they must form a string of the form  \[\ b \ (b-1) \cdots 1 \ \ (n-a) \ (n-a-1) \cdots (b+1)\]
for some $b<n-a$.
\end{proof}

Note that Lemma \ref{layers} implies that the total number of permutations (cyclic or not) of length $n$ that avoid both 123 and 231 is $1+{n\choose 2}$.

Another way to state the result of Lemma \ref{layers} is that if $p$ is a cyclic permutation that avoids 123 and 231, then
\begin{equation} \label{rules}  p_i= \left\{ \begin{array}{l@{\ }l}
n+1-i \hbox{  if $1\leq i\leq a$},\\
 a+b+1-i \hbox{  if $a+1\leq i\leq a+b$, and  }
\\n+b+1-i \hbox{ if $a+b+1\leq i\leq n$.}
\end{array}\right.
\end{equation}

We will use the identities stated in (\ref{rules}) in the rest of this section without referencing (\ref{rules}) each time.

We will call the permutation $p$ defined by the triple $(a,b,c)$ the permutation {\em of} that triple.
 We will call a triple $(a,b,c)$ a {\em good} triple if its permutation is cyclic. 

Now we are going to prove some results, mostly necessary conditions, regarding the parameters $a$, $b$, and $c$ of good triples.

\begin{proposition} \label{conjugates}
The triple $(a,b,c)$ is good if and only if the triple $(a,c,b)$ is good.
\end{proposition}

\begin{proof} It suffices to show that the permutations of those two triples are conjugates of each other, since that implies that they have the same cycle structure. In order to see that the permutation $p$ of the triple $(a,b,c)$ and the
permutation $q$ of the triple $(a,c,b)$ are conjugates, let $w$ be the {\em decreasing} permutation of length $n$.

For the rest of this paper, we will multiply permutations {\em left to right}, so $rs$ means that we first apply the permutation $r$ to the set $[n]$, then we apply the permutation $s$ to that set.

Then $p=wx$, where $x(i)=i$ if $i>b+c$, and $x(i)=i-b$ (modulo $b+c$) if $i\leq b+c$. (So $x$ cyclically rotates the string of the last $b+c$ entries of $w$ forward by $b$ positions.)
On the other hand, $q=w x^{-1}$, since $x^{-1}$ rotates that same string backward by $b$ positions, which is the same as rotating it forward by $c$ positions.

Note that $w$  is an involution, so $q^{-1}=xw$, and so $wq^{-1}w=w(xw)w=wx=p$. Therefore, $p$ is a conjugate
of $q^{-1}$, and therefore, of $q$.
\end{proof}

\begin{proposition} \label{upper}
If the triple $(a,b,c)$ is good, then $a\leq \lfloor n/2 \rfloor $ and, $c\leq \lfloor n/2 \rfloor $.
\end{proposition}

\begin{proof} If $n=2k+1$, and $a\geq k+1$, then $p_{k+1}=k+1$ is a fixed point. If $n=2k$, and $a\geq k+1$,
then $p_k=k+1$ and $p_{k+1}=k$ form a 2-cycle.

Similarly, assume that $c> \lfloor n/2 \rfloor$. Then the third layer of $p$ starts in position $n-c+1$, in the entry
$b+c$. So at that position, the entry in the position is larger than the index of the position. Moving to the right one
position at a time, the index of the position will increase by 1 at each step, while the entry in the position will decrease
by one. At the end, we will be at position $n$, that will contain the entry $b+1$. So at the end, the index of the position is
larger than the entry in it. As both the index and the content of our position changed one by one, there had to be
a leftmost position $j$ where the index $j$ was at least as large as the entry $p_j$. If, at that point, equality held, then
$j=p_j$ is a fixed point $p$. If, on the other hand, at that point $j>p_j$ held, then $p_j=j-1$, and therefore, $p_{j-1}=j$, and
$(j-1 \ j)$ is a 2-cycle in $p$.
\end{proof}

The following corollary is a direct consequence of Propositions \ref{conjugates} and \ref{upper}. 

\begin{corollary}
If the triple $(a,b,c)$ is good, then $b\leq \lfloor n/2 \rfloor$.
\end{corollary}

\begin{proposition}
If the triple $(a,b,c)$ is good, then $a\geq b$, and $a\geq c$.
\end{proposition}

\begin{proof} It follows from Proposition \ref{conjugates} that it suffices to prove $a\geq b$. Let us assume the contrary,
that is, that $b\geq a+1$.
Consider  the second layer of $p$. Its first entry is in  position $a+1$, and it is $b$. So $p_{a+1}=b$, then $p_{a+2}=b-1$, and this trend continues,  ending in $p_{a+b}=1$.  If $a+1=b$, then $a+1$ is a fixed point in $p$. If not, then, 
 sequence of entries $b, b-1,...,1$ starts {\em above} the sequence of positions $a+1,a+2,...a+b$, but ends below it, so it crosses it somewhere, and then the proof is identical to that of the inequality $c\leq \lfloor n/2 \rfloor $ in Proposition \ref{upper}.
\end{proof}

\subsubsection{Restrictions related to common divisors of layer lengths}

It turns out that $b$ and $c$ cannot have large common divisors.
\begin{lemma} \label{cdivisors} If $(a,b,c)$  is a good triple, then the largest common divisor of $b$ and $c$ is 1 or 2.
Furthermore, if the largest common divisor of $b$ and $c$ is 2, then $a$ is even.
\end{lemma}

\begin{proof}
Let us assume the contrary, that is, that $b=fk$ and $c=gk$, with $k>2$.
The crucial observation is that in this case, $p$ permutes the {\em remainder classes} modulo $k$. In fact, we claim
that for all $i$, the equality
\begin{equation} \label{modk} p_i \equiv n+1 - i  \qquad{(\hbox{mod } k}) \end{equation}
holds. In order to prove (\ref{modk}), first note that it holds for $i=1$, since $p_1=n$. Now we show that (\ref{modk})
remains true for each index $i$, as we grow $i$ one by one.  First, note that (\ref{modk}) stays
true as long as $i\leq a$.  That is, note that  (\ref{modk}) stays
true while we are on the first layer, since every time we make one step to the right, both sides decrease by 1. When we pass from the first layer to the second, $i$ grows from $a$ to $a+1$, while $p_i$ decreases from
$n-a+1=(f+g)k+1$ to $b=fk$, so modulo $k$, it decreases by 1. So (\ref{modk}) remains true. After this, (\ref{modk})
remains true at each step to right one the second layer (since again, each step decreases both the left-hand side and the
right-hand side by 1). When we pass from the second layer to third, $i$ changes from $a+b$ to $a+b+1$, while
$p_i$ changes from 1 to $b+c=(f+g)k$, so modulo $k$, it decreases by 1. Finally, (\ref{modk}) remains true on the third
layer as it did on the first two layers.

Equality (\ref{modk}) shows that $p$ acts as an {\em involution} on the remainder classes modulo $k$. In particular, if
the equation $j \equiv n+1-j$ (modulo $k$), or, equivalently, $j \equiv a+1-j$ (modulo $k$),  has a solution $j$, then the remainder class of $j$ is mapped
onto {\em itself} by $p$. In other words, that remainder class is a union of cycles, so $p$ cannot be cyclic.

If the equation $j=a+1-j$ does not have a solution modulo $k$, then select any remainder class $i$, and the remainder
class $a+1-i$. These two classes are mapped onto each other, so they form a union of cycles in $p$. This union does not
contain all of $p$, since $p$ has $k>2$ remainder classes. So again, $p$ cannot be cyclic.

Finally, if $k=2$, and $a$ is odd, then the remainder class 1 maps onto itself. In other words, odd entries map into odd
entries, and even entries map into even entries, so $p$ is not cyclic.
\end{proof}

Lemma \ref{cdivisors} stops short of claiming that $b$ and $c$ must always be relatively prime to each other. The next proposition shows that
in some cases, they have to be. In the next section, we will see that those cases are not as rare as it might now seem.

\begin{proposition}
If $(a,b,c)$ is a good triple and $a=b+c$, then $b$ and $c$ are relatively prime to each other.
\end{proposition}

\begin{proof} Let us assume the contrary, that is, that $b$ and $c$ are both even numbers. Then so is $b+c=a$.
Furthermore,  $n=2a$ is even, so $p$ can only be cyclic if it is an {\em odd} permutation, that is, if it has
an odd number of inversions. On the other hand,  the number of inversions of $p$ is
\begin{equation} \label{inversions} I_{a,b,c}={a\choose 2}+{b\choose 2}+ {c\choose 2}+a(b+c).
\end{equation}  Note that if $x$ is an even number, then ${x\choose 2}$ is
odd if and only if $x=4k+2$ for some integer $k$. As $a=b+c$, this must hold for an even number of summands
out of the first three summands of $I_{a,b,c}$. As $a(b+c)$ is always even, it follows that $I_{a,b,c}$ is always even.
\end{proof}

\subsection{The size of the first layer}
The following lemma is probably the most suprising result of this paper. We have already seen in Proposition \ref{upper} that if $(a,b,c)$ is a good triple, then $a\leq n/2$. Interestingly, $a$ cannot be much smaller either.

\begin{lemma} \label{layerlengths} Let $(a,b,c)$ be a good triple, and let us assume that $b\leq c$. Recall that $n=a+b+c$. Then
\begin{enumerate}
\item if $n$ is even, then $a=n/2$ or $a=(n/2)-1$, and
\item if $n$ is odd, then $a=(n-1)/2$.
\end{enumerate}
\end{lemma}

In other words, we never have more than two choices for $a$. This immediately proves the crude upper bound $C_n(123,231)\leq n$, since we never have more than $n/2$ choices for $b$.

In the rest of this paper, we will often consider  $p$ to be a directed path. For instance, if $p=2413$, then $p_1=2$,
$p_2=4$, $p_3=1$ and $p_4=3$, and so $p$ is a directed path that goes from 1 to 2 to 4 to 3 and then back to 1.
We also say that $p$ maps 1 to 2, 2 to 4, 4 to 3, and 3 to 1.

\begin{proof}
Let us assume the contrary, that is, that $a\leq b+c-3$. Note that as $a\geq c$, this implies that $b\geq 3$.
 Let $a=kb+r$, with $0\leq r\leq b-1$.

 We will show that the entry
$r+1$ of $p$ is part of a cycle that is shorter than $n$. In fact, we will show that it is part of a cycle that does not even
contain all entries of the second layer.

 Note that  because $r+1\leq b\leq c\leq a$, the equality $a=r+1$ could only hold if $a=b=c$ held, but that would imply that $r=0$, and $a=b=c=1$, contradicting the assumption that
$a\leq b+c-3$.  As $r+1\leq b$,  the entry $r+1$ is on the second layer of $p$, and $p$ maps it
to the first layer, to $p_{r+1}=n-r$. From there, $p$ continues to $p_{n-r}=n+b+1-(n-r)=b+r+1$ on the last layer, then to $p_{b+r-1}=n-b-r$ on the first layer, then again to $p_{n-b-r}=2b+r+1$ on the last layer, and so on. The important point is
that $0\leq a-c<b-2<b$, so $p$ will visit the last layer as many times ($k$ times) as the first layer before running out of space and
returning to the second layer. The last visit to the last layer before the first return to the second layer will be at $p_{n-r-(k-1)b}=kb+r+1=a+1$. From there,
$p$ goes to $p_{a+1}=b$. If $b$ happens to equal $r+1$, then we can stop, as we have just found a cycle that contains
only one entry form the second layer.

Otherwise, we follow $p$ a bit further. Next, $p$ goes to $p_b=n+1-b$ on the first layer, then to $p_{n+1-b}=2b$ on the last layer, and so on, making $k$
visits on each of the first and last layers. The last visit on the last layer will be at $p_{n+1-kb}=(k+1)b=kb+r+(b-r)=
a+(b-r)$. Finally, from here, we move on to the second layer, to $p_{a+b-r}=r+1$, where we started our walk. So we
have found a cycle in $p$ that contains only two entries, $r+1$ and $b$, from the second layer, completing our proof.
\end{proof}

\begin{example} {\em Let $a=7$, let $b=3$,  and let $c=7$. Then $n=17$, and }
\[p = 17 \ 16 \ 15 \ 14 \ 13 \ 12 \ 11 \ 3 \ 2 \ 1 \ 10 \ 9 \ 8 \ 7 \ 6 \ 5 \ 4 .\]

{\em We have $a=2b+1$, so $r=1$. Starting at $r+1=2$, the path of $p$ goes from 2 to 16 to 5 to 13 to 8  to 3 to
15 to 6 to 12 to 9 to 2, completing a cycle that contains only two entries from the second layer.}
\end{example}

\subsection{Positive results}
In this section, we will assume without loss of generality that $b\leq c$, unless stated otherwise. Our results will be positive, that is, they will show
that if $a$, $b$, and $c$ satisfy certain necessary conditions, then the triple $(a,b,c)$ is good.

\begin{theorem} \label{aisb+c}
Let $(a,b,c)$ be a triple of positive integers that satisfies $a=b+c=n/2$, with $b$ and $c$ relatively prime to each other.
Then $(a,b,c)$ is a good triple.
\end{theorem}

\begin{proof}  Let $a=kb+r$, with $1\leq r\leq b$,  then $c=(k-1)b+r$, and
$n=a+b+c=2a=2kb+2r$.  Note that this implies that $b$ and $r$ are relatively prime to each other. Indeed,
if $b=xd$ and $r=yd$ held for some $d>1$, then  $c=(k-1)b+r$ would also be divisible by $d$, which is a contradiction.

 Let us start following $p$, beginning at any entry $i$ on the second layer. So
$1\leq i\leq b$.  From that entry, $p$ goes to $p_i=n+1-i$,
then to $p_{n+1-i}=b+i$, then to $p_{b+i}=n+1-b-i$, and so on. In the first layer, $p$ will visit positions $i$, $b+i$, $2b+i$, and
so on, while on the last layer, $p$ will visit positions $n$, $n-b$, $n-2b$, and so on.
Continuing in this way, $p$ will visit all entries of the first layer whose position index is congruent to $i$ modulo $b$, and
all entries on the last layer whose position index (when counted from the {\em right}) is congruent to $i$ modulo $b$
before returning to the second layer. Therefore, in order to prove that $p$ is cyclic, {\em it suffices to prove that $p$ contains
all $b$ entries of its second layer in one cycle}. Indeed, we have just seen that between two visits to the second layer,
$p$ covers an entire remainder class of positions on the first and third layers. So if a cycle contains all entries of the second
layer, then that cycle contains all entries of $p$.

Crucially, as the first layer is $b$ units
longer than the last layer, $p$ will run out of space on the last layer first. In other words, $p$ will always arrive at the second layer from the first layer.

 That is, if $i\leq r$, then as $p$ arrives at position $kb+i$
on the first layer, it finds the entry $p_{kb+i}=n+1-kb-i=a+r+1-i$ there, and then it goes to the {\em second} layer, to the entry
$p_{a+r+1-i}=b+i-r$. If $i>r$, then as $p$ arrives at position $(k-1)b+i$ on the first layer, it finds the entry $p_{(k-1)b+i}=
a+b+r+1-i$ there, and then it goes to the {\em second} layer, to the entry $p_{a+b+r+1-i}=i-r$.

So in all cases, the first entry that $p$ visits on the second layer after visiting $i$ is the entry that is congruent to
$i-r$ modulo $b$. In other words, each visit of the second layer occurs $r$ spots to the right of the last one, modulo $b$.

However, that implies that $p$ will visit all its entries on the second layer before returning to its starting point $i$, since
$r$ is relatively prime to $b$, the length of the second layer.
\end{proof}

\begin{example} {\em Let $a=7$, let $b=3$, and let $c=4$. Then $n=14$, and }
\[p= 14 \ 13 \ 12 \ 11 \ 10 \ 9 \ 8 \ 3 \ 2 \ 1 \ 7 \ 6 \ 5 \ 4 .\]
{\em Let us start at $i=1$. Then $p$ goes from 1 to 14 to 4 to 11 to 7 to 8 to 3 to 12 to 6 to 9 to 2 to 13 to 5 to 10 to 1.
See Figure \ref{for2-11} for an illustration. Note that for each pair of layers $(L,M)$, we represented all movements from $L$ to $M$ using the same kind of arrows, and still we only needed four arrow types instead of the potential maximum, nine. 
This shows the relative simplicity of the action of $p$.

\begin{figure}[htbp]
 \begin{center}
  \includegraphics[width=60mm]{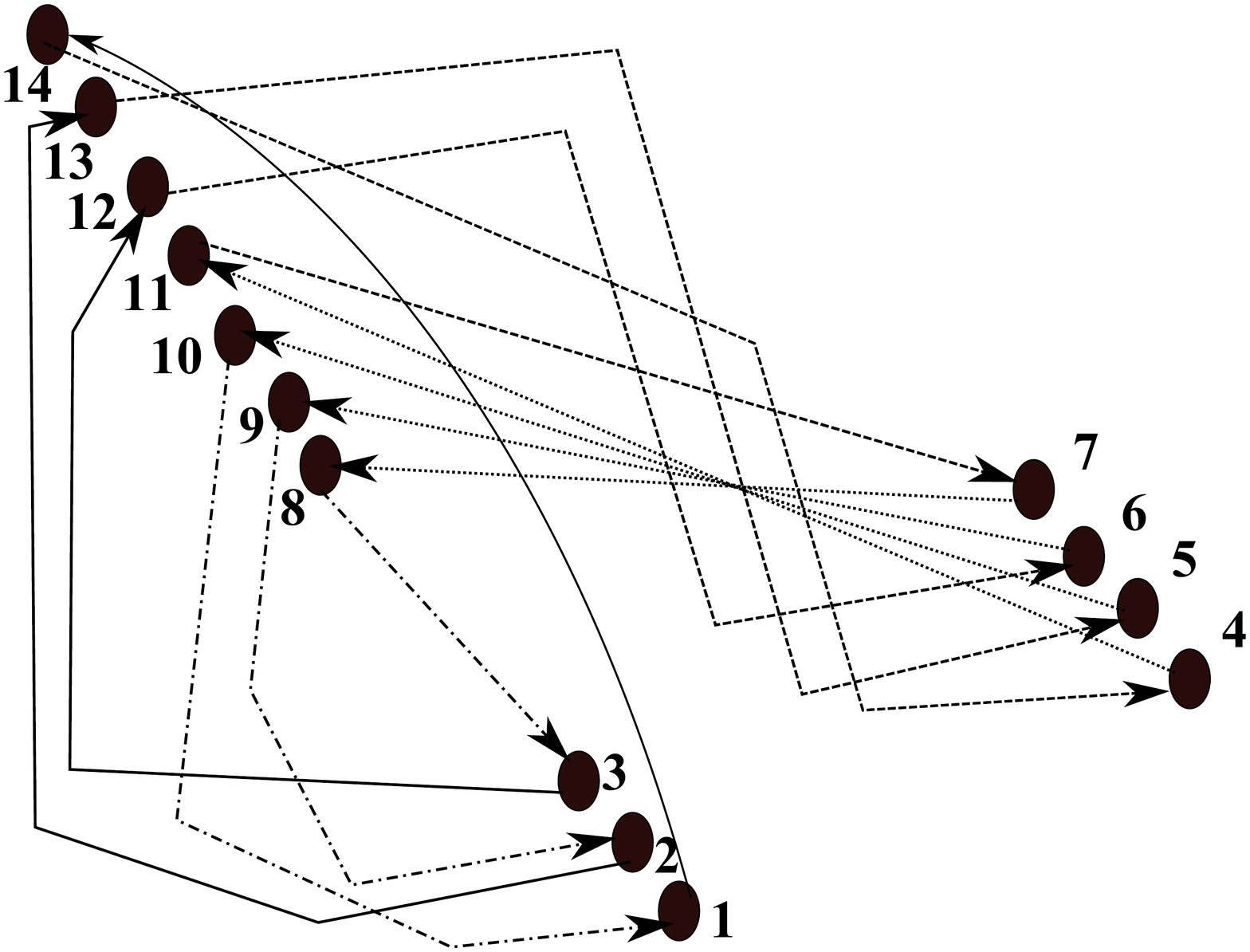}
  \label{for2-11}
\caption{The action of $p= 14 \ 13 \ 12 \ 11 \ 10 \ 9 \ 8 \ 3 \ 2 \ 1 \ 7 \ 6 \ 5 \ 4 $. }
 \end{center}
\end{figure}

Note that between the first and second visit to the second layer, $p$ visits all entries $p_j$ on the first layer where $j=3\ell +1$, and all entries on the last layer that are in position $3\ell +1$ when counted from the end. Between the second and third visits of $p$ to the second layer, the same goes for entries in positions $3\ell$, and between the third and fourth visits of $p$ to the second layer, the same goes for entries in positions $3\ell +2$. }
\end{example}

\begin{theorem} \label{diff1}
Let $(a,b,c)$ be a triple of positive integers that satisfies $a=b+c-1$, with $a\geq c\geq b$, and with $b$ and $c$ relatively
prime to each other.
Then $(a,b,c)$ is a good triple.
\end{theorem}

\begin{proof}
  Let $a=kb+r$, with $0\leq r\leq b-1$,  then $c=(k-1)b+r+1$, and
$n=a+b+c=2a+1=2kb+2r+1$.

The proof is similar to that of Theorem \ref{aisb+c}, with one significant difference. This time, $c=a-b+1$, so there is
exactly one entry on the second layer of $p$, namely the entry $r+1$, so that if we start walking the along the path of
$p$ at $r+1$, then we will visit the last layer as many times as the first layer before returning to the second layer. (In the situation of Theorem  \ref{aisb+c}, there was no such entry.) Otherwise, just as in the proof of Theorem \ref{aisb+c},
it suffices to show that $p$ contains a cycle that contains all entries of the second layer.

So let us start walking at this exceptional entry $r+1$. Our walk takes us to position $r+1$ that contains the entry
$p_{r+1}=n-r$, then position $n-r$, that contains entry $p_{n-r}=b+1-r$, and so on, through positions $r+1$, $b+r+1$,
$2b+r+1$, and so on on the first layer, and positions $n-r$, $n-r-b$, and so on on the last layer, eventually reaching the
leftmost position of the last layer, position $n-r-(k-1)b=a+b+1$, containing the entry $p_{a+b+1}=n-a=a+1$. We will
then reach the second layer in the next step, at its leftmost position, at $p_{a+1}=b$.

Other than the exceptional entry $r+1$, all entries of the second layer (including $b$) will behave identically. That is, from $i$,  the walk of $p$ goes to $p_i=n+1-i$,
then to $p_{n+1-i}=b+i$, then to $p_{b+i}=n+1-b-i$, and so on. In the first layer, $p$ will visit positions $i$, $b+i$, $2b+i$, and
so on, while on the last layer, $p$ will visit positions $n$, $n-b$, $n-2b$, and so on.  If $i\leq r$, then as $p$ arrives at position $kb+i$
on the first layer, it finds the entry $p_{kb+i}=n+1-kb-i=a+r+2-i$ there, and then it goes to the {\em second} layer, to the entry
$p_{a+r+2-i}=b+i-r-1$. If $i>r$, then as $p$ arrives at position $(k-1)b+i$ on the first layer, it finds the entry $p_{(k-1)b+i}=
a+b+r+2-i$ there, and then it goes to the {\em second} layer, to the entry $p_{a+b+r+2-i}=i-r-1$. So in all cases when
$i\neq r+1$, the next visit on the second layer after visiting entry $i$ is at the unique entry that is congruent to
$i-r-1$ modulo $b$.  In other words, each visit of the second layer occurs $r+1$ spots to the right of the last one, modulo $b$.

Our proof is now complete noting that $b$ is relatively prime to $r+1$. Indeed, if $d>1$ divides both $b$ and $r+1$, then
it also divides $c=(k-1)b+r+1$, contradicting the assumption that $b$ and $c$ are relative primes.
\end{proof}

\begin{example} {\em Let $a=6$, let $b=3$, and let $c=4$. Then $n=13$, and }
\[p=  13 \ 12 \ 11 \ 10 \ 9 \ 8 \ 3 \ 2 \ 1 \ 7 \ 6 \ 5 \ 4 .\]
{\em As $a=2b$, we have $r=0$, and so we start at $r+1=1$. Then $p$ goes from 1 to 13 to 4 to 10 to 7 to $b=3$ to 11 to
6 to 8 to $b-(r+1)=2$ to 12 to 5 to 9 to $b-2(r+1)= 1$. }
\end{example}

The case when $a=b+c-2$ is a little bit more cumbersome. Therefore, we need  two more negative results before announcing our enumeration formulas. For these two propositions, we drop the assumption that $b\leq c$.

\begin{proposition} \label{divby4} Let $n=4k$. If $(a,b,c)$ is a good triple, then $a=2k=n/2$.
\end{proposition}

\begin{proof} All we need to show is that it is not possible to have $a=2k-1$, and $b+c=2k+1$.  Let us assume that
that is the case; in particular, that both $a$ and $b+c$ are odd, and therefore, exactly one of $b$ and $c$ is even. As we no longer assume that $b\leq c$, we can
assume without loss of generality that $b$ is even and $c$ is odd.  The number $I_{a,b,c}$ of inversions of $p$ is given in (\ref{inversions}). It follows from our assumption that $a(b+c)$ is odd. As $p$ is a permutation of even length, if it is cyclic, then it has to have an odd number of inversions. Therefore, the sum $i_{a,b,c}={a\choose 2} + {b\choose 2} +{c\choose 2}$
has to be even. Recall that we can assume that $b$ is even and $c$ is odd. There are the following two cases.
\begin{enumerate}
\item If $b=4\ell$, then $a\equiv c-2$ modulo 4, so ${a\choose 2}+{c\choose 2}$ is odd, while ${b\choose 2}$ is even.
So   $i_{a,b,c}$ is odd, and therefore, $I_{a,b,c}$ is even.
\item If $b=4\ell+2$, then $a\equiv c$ modulo 4, so ${a\choose 2}+{c\choose 2}$ is even, while ${b\choose 2}$ is odd.
So again,   $i_{a,b,c}$ is odd, and therefore, $I_{a,b,c}$ is even.
\end{enumerate}
So $p$ cannot be cyclic if $a=2k-1$ and $b=2k+1$, proving our claim.
\end{proof}

\begin{proposition} \label{4k+2} Let $n=4k+2>2$, and let $a=2k$. Then $b$ and $c$ must both be even. \end{proposition}

\begin{proof} Let us assume the contrary, that is, that $b$ and $c$ are both odd (they must be of the same parity, since
$b+c=n-a=2k+2$).

If $p$ is cyclic, then it has an odd number of inversions. As $a(b+c)$ is even, that means that $i_{a,b,c}={a\choose 2} + {b\choose 2} +{c\choose 2}$ must be odd. There are again two cases.
\begin{enumerate}
\item If $a$ is divisible by 4, then $b+c$ is not, so, given that $b$ and $c$ are both odd, $b\equiv c$ modulo 4,
so ${b\choose 2} +{c\choose 2}$ is even, and so is ${a\choose 2}$, implying that $i_{a,b,c}$ is even.
\item If $a=4\ell +2$, then $b+c$ is divisible by four, so $b\equiv c-2$ modulo 4, so ${b\choose 2} +{c\choose 2}$ is odd,
and so is ${a\choose 2}$,  implying again that $i_{a,b,c}$ is even.
\end{enumerate}
So if $b$ and $c$ are odd, then $I_{a,b,c}$ is even, and $p$ cannot be cyclic.
\end{proof}

\begin{theorem} \label{diff2} Let $(a,b,c)$ be a triple of positive integers satisfying $a=2K$, with $a\geq b$, and 
$a\geq c$,
so that $b+c=2K+2$, and $b$ and $c$ are both even, and have no common divisor larger than 2. Then $(a,b,c)$ is a good triple.
\end{theorem}

\begin{proof}  Let $a=kb+r$, with $0\leq r\leq b-1$,  then $c=(k-1)b+r+2$, and
$n=a+b+c=2a+2=2kb+2r+2$. Note in particular that the conditions imply that $r$ is an even number,
and that $b$ and $r+2$ have largest common divisor 2.

The proof is similar to that of Theorem \ref{diff1}, with one significant difference. This time, $c=a-b+2$, so there are
exactly two entries on the second layer of $p$, namely the entries  $r+1$ and $r+2$, so that if we start walking the along the path of
$p$ at $r+1$, or at $r+2$, then we will visit the last layer as many times as the first layer before returning to the second layer. (In the situation of Theorem  \ref{aisb+c}, there was no such entry, and in the situation of Theorem \ref{diff1}, there was one such entry.)

So let us start walking at the exceptional entry $r+1$. Our walk takes us to position $r+1$ that contains the entry
$p_{r+1}=n-r$, then to position $n-r$, that contains the entry $p_{n-r}=b+1-r$, and so on, through positions $r+1$, $b+r+1$,
$2b+r+1$, and so on on the first layer, and positions $n-r$, $n-r-b$, and so on on the last layer, eventually reaching the
second-from-the-left position of the last layer, position $n-r-(k-1)b=a+b+2$, containing the entry $p_{a+b+2}=n-a-1=a+1$. We will
then reach the second layer in the next step, at its leftmost position, at $p_{a+1}=b$.

After $r+1$, all entries of the second layer, {\em except for} $r+2$,  will behave identically. That is, from $i$,  the walk of $p$ goes to $p_i=n+1-i$,
then to $p_{n+1-i}=b+i$, then to $p_{b+i}=n+1-b-i$, and so on. In the first layer, $p$ will visit positions $i$, $b+i$, $2b+i$, and
so on, while on the last layer, $p$ will visit positions $n$, $n-b$, $n-2b$, and so on.  If $i\leq r$, then as $p$ arrives at position $kb+i$
on the first layer, it finds the entry $p_{kb+i}=n+1-kb-i=a+r+3-i$ there, and then it goes to the {\em second} layer, to the entry
$p_{a+r+3-i}=b+i-r-2$. If $i>r+2$, then as $p$ arrives at position $(k-1)b+i$ on the first layer, it finds the entry $p_{(k-1)b+i}=
a+b+r+3-i$ there, and then it goes to the {\em second} layer, to the entry $p_{a+b+r+3-i}=i-r-2$. So in all cases when
$i\notin \{r+1,r+2\}$, the next visit on the second layer after visiting entry $i$ is at the unique entry that is congruent to
$i-r-2$ modulo $b$.  In other words, each visit of the second layer occurs $r+2$ spots to the right of the last one, modulo $b$.

Therefore, the visits of $p$ at the second layer will occur in the following order:  $r+1, b, b-(r+2), b-2(r+2), \cdots $,
understood modulo $b$. As $b$ and $r+2$ have largest common divisor 2, the first $b/2$ visits starting with $b$ will all
be at distinct even entries of the second layer, the last one arriving at $r+2$.

The entry $r+2$ is exceptional in the same way as $r+1$ is -- the walk starting there will reach the second layer from the
third layer, not the first. Indeed, our walk takes us to the position $r+2$ that contains the entry
$p_{r+2}=n-r-1$, then to position $n-r-1$, that contains the entry $p_{n-r-1}=b+2-r$, and so on, through positions $r+2$, $b+r+2$,
$2b+r+2$, and so on on the first layer, and positions $n-r-1$, $n-r-b-1$, and so on on the last layer, eventually reaching the
leftmost position of the last layer, position $n-r-1-(k-1)b=a+b+1$, containing the entry $p_{a+b+1}=n-a$. We will
then reach the second layer in the next step, at the entry $p_{n-a}=a+b+1-(n-a)=a+b+1-(b+c)=a-c+1=b-1$.

After this, the remaining entries of the second layer again behave identically, just as we have seen two paragraphs above.
So the next visits on the second layer are at $r+1, b-1, b-1-(r+2), b-1-2(r+2), \ldots $. As $b$ and $r+2$ have largest
common divisor 2, the smallest solution of the equation $b-1 -j(r+2) \equiv 1$, or, equivalently, $-j(r+2)\equiv 0$ modulo $b$ is $j=b/2$. So the first $b/2$ visits will be at distinct odd entries of the second layer, and the next one will be at the entry 1, closing the cycle of $p$.
\end{proof}

\begin{example}
Let $a=8$, $b=4$, and $c=6$. Then $n=18$, $r=0$, and
\[p = 18 \ 17 \ 16 \ 15 \ 14 \ 13 \ 12 \ 11 \ 4 \ 3 \ 2 \ 1 \ 10 \ 9 \ 8 \ 7 \ 6 \ 5  .\]
Starting at $r+1=1$, the permutation $p$ maps 1 to 18 to 5 to 14 to 9 to $b=4$ to 15 to 8 to 11 to $r+2=2$ to  17 to 6 to 13
to 10 to $b-1=3$ to 16 to 7 to 12 to 1.
\end{example}

Now that we have completely characterized good triples, and therefore, cyclic permutations that avoid both 123 and 231,
we are ready to announce our enumeration formulas.

\begin{theorem} \label{main} Let $\phi$ be the Euler totient function. That is, for a positive integer $z$, let $\phi(z)$ be
 the number of  positive integers less than $z$ that are relatively prime to $z$. Then, for $n>2$, 
the following enumeration formulas hold.
\begin{equation}  C_n(123,231)= \left\{ \begin{array}{l@{\ }l}
\phi(2k) =\phi(n/2) \hbox{  if $n=4k$},\\
\phi(k+1)+\phi(2k+1) =\phi\left(\frac{n+2}{4}\right) + \phi\left(\frac{n}{2} \right) \hbox{ if $n=4k+2$,}\\
\phi(m) = \phi((n+1)/2) \hbox{  if $n=2m-1$.  }
\end{array}\right.
\end{equation}
\end{theorem}

Note that  the formula of Theorem \ref{main} does not hold for $n=2$, because in that case, $k=0$, so $k+1=2k+1$, 
resulting in an overcount and leading to the incorrect value of 2, instead of the correct value of $C_2(123,231)=1$.
The sequence  $C_n(123,231)$ is listed in OEIS \cite{oeis} as sequence A309563.

\begin{proof}
If $n=4k$, and $(a,b,c)$ is a good triple, then Proposition \ref{divby4} shows that  $a=2k$.  Let $b<a$ be relatively prime
to $a$. Then, and only then,  $b$ is also relatively prime to $c=a-b$. Theorem \ref{aisb+c} then shows
that $(a,b,c)$ is a good triple. Therefore, there are $\phi(2k)$ choices for $b$, and hence, for a good triple
$(a,b,c)$.

If $n=4k+2$, then Lemma \ref{layerlengths} shows that either $a=2k+1$ or $a=2k-1$ holds for all good triples $(a,b,c)$. In the first case,   $b+c=2k+1$.    Let $b<a$ be relatively prime
to $a$. Then, and only then,  $b$ is also relatively prime to $c=a-b$.  By Theorem \ref{aisb+c}, all such choices of $b$ will lead to a valid triple.
So this case contributes $\phi(2k+1)$ permutations to the total count.
In the second case, $b+c=2k+2$. Proposition \ref{4k+2} shows that $b$ and $c$ must both be even, implying that
 $b/2 \in [k]$, since $c\neq 0$.  Lemma \ref{cdivisors} shows that $b/2$ and $c/2$ must be relatively prime to each other,
and that is equivalent to saying that $b/2$ is relatively prime to $(b+c)/2=k+1$. On the other hand,  Theorem \ref{diff2} shows that if $b/2$ and $c/2$  are relatively prime to each other, then the triple $(a,b,c)$ is good. So this second case contributes $\phi(k+1)$  permutations to the total count.

Finally, if  $n=2m-1$, then by Lemma \ref{layerlengths} we must have $a=m-1$, $b+c=m$, and $b$ must be any positive integer less than $m$ that is relatively prime to
$b+c=m$. By Theorem \ref{diff1}, each such choice of $b$ will lead to a valid triple. This completes the proof.
\end{proof}

Recall that after the proof of Lemma \ref{layerlengths}, we pointed out that the crude upper bound $C_n(123,231)\leq n$
holds. Now, using Theorem \ref{main} and the trivial inequality $\phi(z)\leq z-1$, we can sharpen that bound to
$C_{n}(123,231) \leq \frac{3n-6}{4}$, for $n\geq 4$. This upper bound is attained for integers $n=4k+2$ if and only if both
$k+1$ and $2k+1$ are primes. For instance, $k=6$, yielding $n=26$, satisfies this requirement.

\section{The pair $(123,132)$}

It turns out that the enumeration formula for the pair $(123,132)$ is significantly simpler.

\begin{theorem}
For all $n\geq 3$, the equality \[C_n(123,132)=2^{\lfloor (n-1) /2 \rfloor }\] holds.
\end{theorem}

\begin{proof}
First note that in any permutation that avoids both 123 and 132, (so not only in the cyclic permutations avoiding those patterns), the entry 1 must be in the last or next-to-last position. Once the place of the entry 1 is chosen,  the entry 2 has to be in the last or next-to-last available position, and so on.

 Let $p=p_1p_2\cdots p_n$ be a cyclic permutation that avoids both 132 and 123. Then $p_1=n$ or $p_1=
n-1$, otherwise $p_1$ is eventually followed by two larger entries, forcing $p_1$ to be the first entry of a 123-pattern
or a 132-pattern.  If $p_1=n$, then $p_n\neq 1$, since that would mean that $(1n)$ is a 2-cycle of $p$, and if
$p_1=n-1$, then $p_{n-1}\neq 1$, since then $(1(n-1))$ would be a 2-cycle in $p$.  So, in both cases, exactly one of the two  positions that were originally eligible to contain the entry 1  is available. Once we selected the position of 1, we return to the front of the permutation,
and select the value of $p_2$. We have two choices for $p_2$, namely the two largest remaining entries. The choice that we make for $p_2$ will eliminate one of these two positions for the entry 2.

We must show that this process keeps going on like this, that is,  if $i\leq \lfloor (n-1) /2 \rfloor $, then we will always have
two choices for the entry $p_i$, while if $i<\lfloor (n-1) /2 \rfloor $, we will have exactly one choice. This will prove the
statement of the theorem.

As we fill the positions of $p$, we form two sets. Let $S_i=\{p_1,p_2,\cdots ,p_i\}$, and let
$T_i$ be  set of positions that the entries $1,2,\cdots ,i$ occupy. Note that if $1\leq i\leq  \lfloor (n-1) /2 \rfloor $,
then $S_i \neq T_i$. Indeed, $S_i =T_i$ would imply that the restriction of $p$ to $S_i$ is a bijection from $S_i$ onto
itself, that is, it is a permutation of the set $S_i$, and therefore, a union of its cycles. That would contradict
the condition that the longer permutation $p$ itself is a cycle.

It is easy to see that both $S_i$ and $T_i$ are $i$-element subsets of the set $\{n-i,n-i+1, \cdots ,n\}$. (In the first
two paragraphs of this proof, we show this for $S_1$ and $T_1$, and the cases of general $i$ are very similar to these.)
 In other words, they both contain at most one gap, either inside, or at the end.

We will now count the ways in which the sets $S_1,S_2,\cdots $ and $T_{1},T_{2}, \cdots $ can be built up
from $S_0=T_0=\emptyset$.

When we extend $S_{i-1}$ to $S_i$, we do one of two things.  Either we fill the gap in $S_{i-1}$, turning it into
the interval $[n-i+1,\cdots ,n]$, or we add a new, minimal element $n-i$ to $S_{i-1}$ to form $S_i$.

When we extend $T_{i-1}$ to $T_i$, we could think that we have these same two choices. However, we cannot make the same choice as we made for $S_i$. Indeed, if $S_i$ and $T_i$ are both equal to the interval $[n-i+1,\cdots ,n]$, then
$S_i=T_i$, which we have already excluded. If $S_i$ and $T_i$ were both obtained by adding the new element $n-i$,
that means that $p_i=n-i$ and $p_{n-i}=i$, so $(i \ (n-i) )$ is a 2-cycle, which is a contradiction. Therefore, we have
only one choice when we extend $T_{i-1}$ to $T_i$.

If $n$ is odd, then we make $2^{(n-1)/2}$ choices in this way, building the sequence $S_1,S_2,\cdots ,S_{(n-1)/2}$, and so  selecting the leftmost $(n-1)/2$ entries of $p$, and the
positions of the smallest $(n-1)/2$ entries of $p$. Then we put the last remaining entry in the last remaining position.

If $n=2k$, then we make  $2^{k-1}$ choices in this way, selecting the sequence $S_1,S_2,\cdots, S_{n/2}$, and so
selecting the first $k-1$ entries of $p$, and the
positions of the smallest $k-1$ entries of $p$. This leaves two empty positions, one of them is the $k$th position, and
the other one is somewhere in the second half of $p$. This also leaves two unused entries, one of them is the entry $k$,
and the other one is an entry from the larger half of $p$. As $p(k)\neq k$, our hands are tied, and the proof is complete.
\end{proof}

\section{The remaining pairs} \label{others}

A result of Archer and Elizalde \cite{archer} shows that
\[C_n(132,231)=\frac{1}{2n} \sum_{d|n \atop d=2k+1} \mu(d) 2^{n/d},\]
where $\mu$ is the number theoretical M\"obius function.

Most of the remaining pairs are straightforward to enumerate. This is the content of the next theorem.

\begin{theorem}
The following equalities hold.
\begin{enumerate}
\item For $n\geq 5$, $C_n(123,321)=0$.
\item For $n\geq 3$, $C_n(231,312)=0$.
\item For all positive integers $n$, $C_n(231,321)=1$.
\item For all positive integers $n$, $C_n(132,321)=\phi(n)$.
\end{enumerate}
\end{theorem}

\begin{proof}
\begin{enumerate}
\item The famous Erd\H{o}s-Szekeres theorem shows $S_n(123,321)=0$ if $n\geq 5$. So there are no permutations (cyclic or not) of length five or more avoiding both of those patterns.
\item Permutations that avoid 231 and 312 must start with a decreasing sequence of their $k_1$ smallest entries for some $k_1$,
then continue with a decreasing sequence of their next $k_2$ smallest entries, and so on, like the permutation
$321\ 54 \ 6 \ 987$. (Such permutations are called {\em layered} permutations.) This structure implies that all such
permutations are {\em involutions}, so they cannot be cyclic if $n>2$.
\item In permutations avoiding 231 and 321, the sequence of entries on the right of $n$ must be increasing, and must consist of entries that are larger than the entries on the left of $n$. However, such permutations entries on the left of $n$
are mapped into entries on the left of $n$, which makes it impossible for such a permutation to be cyclic {\em unless} there is nothing on the left of $n$. Therefore, the only cyclic permutation avoiding those patterns is $n12\cdots (n-1)$.
\item If a permutation $p$ of length $n$ avoids both 132 and 321, and does not end in $n$ (which cyclic permutations cannot do),
then it is of the form $ (i+1)\  (i+2) \cdots n \ 1 \ 2 \cdots i$ for some $i\geq 1$. In other words, $p=q^i$, where
$q=23\cdots n1$. Such a permutation is cyclic if and only if $i<n$ is relatively prime to $n$.
\end{enumerate}
\end{proof}

All other pairs of patterns of length three are equivalent to one of those that we have considered, by the trivial symmetries
(taking inverses, or taking reverse complements), except for the pair $(132,213)$. Therefore, the enumeration of cyclic permutations avoiding pairs of
patterns of length three is {\em almost} complete.

\section{Further Directions}
The enumeration of cyclic permutations avoiding a single pattern of length three has proved more difficult than that of pairs of patterns of length three, and no results are yet known.

\begin{table}[h]
\centering
\caption{Number of cyclic permutations avoiding a single pattern of length 3}
\label{one-pattern}
\begin{tabular}{lllllll}
Avoiding: & 123   & 132   & 213   & 231  & 312  & 321   \\
$n$=3     & 2     & 2     & 2     & 1    & 1    & 2     \\
4         & 4     & 4     & 4     & 2    & 2    & 4     \\
5         & 10    & 10    & 10    & 5    & 5    & 10    \\
6         & 24    & 24    & 24    & 12   & 12   & 24    \\
7         & 68    & 68    & 68    & 30   & 30   & 66    \\
8         & 188   & 182   & 182   & 86   & 86   & 178   \\
9         & 586   & 544   & 544   & 253  & 253  & 512   \\
10        & 1722  & 1574  & 1574  & 748  & 748  & 1486  \\
11        & 5492  & 4888  & 4888  & 2274 & 2274 & 4446  \\
12        & 16924 & 14864 & 14864 & 7152 & 7152 & 13468
\end{tabular}
\end{table}

Numerical evidence shown in Table \ref{one-pattern} enabled us to formulate the following conjectures.
\begin{conjecture} For all positive integers $n$, the chain of inequalities
\[C_n(123) \geq C_n(132) = C_n(213) \geq C_n(321) \geq C_n(231) = C_n(312)\]
holds.
\end{conjecture} 

The four distinct sequences in Table \ref{one-pattern} are listed in OEIS as sequences A309504, A309505, A309506, and A309508 \cite{oeis}. 

Note that the equality  $C_n(132) = C_n(213)$ is obvious, since the reverse complement of a cyclic permutation is
a cyclic permutation, and the reverse complement of a $q$-avoiding permutation avoids the reverse complement of $q$.
Also note that the equality   $C_n(231) = C_n(312)$  is obvious, since the inverse of a cyclic permutation is cyclic, and the inverse of a
$q$-avoiding permutation avoids $q^{-1}$. 
\begin{conjecture} \label{twobounds}
For each pattern $q$ of length 3 and $n \geq 3$, the chain of inequalities 
\[2C_n(q) \leq C_{n+1}(q) \leq 4C_n(q) \] holds.
\end{conjecture}

In the following theorem, we prove the lower bound of Conjecture \ref{twobounds} for the pattern 321 (and an infinite 
collection of longer patterns). 

\begin{theorem} 
Let $q=q_1q_2\cdots q_k$ be any {\em involution} of length $k>2$ such that if $q_i=k$, then $i\leq k-2$. 
Then for all $n\geq 2$, the inequality
\[2C_n(q) \leq C_{n+1}(q) \] holds. 
\end{theorem}

Note that 321 is the only pattern of length three that satisfies the requirements of the theorem. There are four such patterns
of length four, namely 4321, 4231, 3412, and 1432. 
\begin{proof}
Let $p=p_1p_2\cdots p_n$ be any cyclic permutation of length $n$ that avoids $q$. Now insert the entry $n+1$ to the next-to-last position
of $p$. Then $p$ is still $q$-avoiding, since $n+1$ is too far back in $p$ to be part of any copies of $q$. Furthemore, 
the obtained permutation $p'$ is still cyclic, since $p_i=p'_i$ for all $i\leq n-1$, and $p$ maps $n$ to $x$, while $p'$
maps $n$ to $n+1$, and then $n+1$ to $x$. So, we get the cyclic diagram of $p'$ by simply inserting the entry $n+1$
between $n$ and $x$ in the cyclic diagram of $p$. 

Doing this for all $C_n(q)$ cyclic, $q$-avoiding permutations of length $n$ yields a set $S$ of  cyclic
$q$-avoiding permutations of length $n+1$, each of which contains the entry $n+1$ in the $n$th position.
 As $q$ is an involution, the inverse $r^{-1}$ of any $q$-avoiding permutation $r$ is also $q$-avoiding. So taking the inverse of each permutation in $S$ yields a set $T$  of $C_n(q)$ cyclic
$q$-avoiding permutations of length $n+1$, each of which contains the entry $n$ in the $(n+1)$st position.
Finally, $S$ and $T$ are disjoint sets, since a cyclic permutation that is longer than 2 cannot contain the 2-cycle $(n \ n+1)$. 
\end{proof}

Some additional numerical evidence raises the following question.
\begin{question} Let $q$ be any pattern of length $k\geq 3$. Is it true that
\[(k-1)C_n(q) \leq C_{n+1}(q) \] if $n\geq  k$? 
\end{question}

Note that for general (non-cyclic) permutations, the answer to the analogous question is a straightforward "yes".  Indeed, 
let the maximal entry $k$ of $q$ be in the $(i+1)$st position of $q$. Then there  are $i$ entries on the left of $k$, and
$ k-1-i$ entries on the right of $ k$ in $q$. Therefore, if $p$ is a $q$-avoiding permutation of length $n$, then  the new maximal entry 
$n+1$ can be inserted in $p$ in $k-1$ ways, so that it is one of the leftmost $i$ entries, or one of the rightmost $k-1-i$ entries. 
In all those cases,  $n+1$ will be either too far left or too far right to be in a $q$-pattern. (See also Exercise 4.33 in
\cite{combperm}.) 
\vskip 0.5 cm

\acknowledgements
\label{sec:ack}
We are grateful to our three anonymous referees whose careful work improved the presentation of our results.

\nocite{*}
\bibliographystyle{abbrvnat}
\bibliography{cyclic-final}

\begin{thebibliography}{7}
\providecommand{\natexlab}[1]{#1}
\providecommand{\url}[1]{\texttt{#1}}
\expandafter\ifx\csname urlstyle\endcsname\relax
  \providecommand{\doi}[1]{doi: #1}\else
  \providecommand{\doi}{doi: \begingroup \urlstyle{rm}\Url}\fi

\bibitem[Archer and Elizalde(2014)]{archer}
K.~Archer and S.~Elizalde.
\newblock Cyclic permutations realized by signed shifts.
\newblock \emph{J. Comb.}, 5:\penalty0 1--30, 2014.

\bibitem[B\'ona(2012)]{combperm}
M.~B\'ona.
\newblock \emph{Combinatorics of Permutations}.
\newblock CRC Press, Boca Raton, FL, 2nd edition, 2012.

\bibitem[B\'ona(2014)]{protected}
M.~B\'ona.
\newblock $k$-protected vertices in binary search trees.
\newblock \emph{Adv. Appl. Math.}, 53:\penalty0 1--11, 2014.

\bibitem[B\'ona(2016)]{awalk}
M.~B\'ona.
\newblock \emph{A Walk Through Combinatorics}.
\newblock World Scientific, Singapore, 4th edition, 2016.

\bibitem[Huang(2019)]{huang}
B.~Huang.
\newblock An upper bound on the number of (132,213)-avoiding cyclic
  permutations.
\newblock \emph{Disc. Math.}, 342\penalty0 (6):\penalty0 1762--1771, 2019.

\bibitem[Sloane()]{oeis}
N.~J.~A. Sloane.
\newblock The online encyclopedia of integer sequences.
\newblock URL \url{www.oeis.org}.

\bibitem[Vatter(2015)]{vatter}
V.~Vatter.
\newblock Permutation classes.
\newblock In M.~B\'ona, editor, \emph{Handbook of Enumerative Combinatorics},
  chapter~12, pages 753--835. CRC Press, Boca Raton, FL, 2015.

\end{thebibliography}
\label{sec:biblio}

\end{document}